\documentclass[12pt,leqno,draft]{article}
\usepackage{amsfonts}
\usepackage{amsmath, amsthm, amsfonts, amssymb, color}
\usepackage{mathrsfs}
\usepackage{color}
\usepackage{stmaryrd}
\makeatletter
\newcommand{\Spvek}[2][r]{%
  \gdef\@VORNE{1}
  \left(\hskip-\arraycolsep%
    \begin{array}{#1}\vekSp@lten{#2}\end{array}%
  \hskip-\arraycolsep\right)}

\def\vekSp@lten#1{\xvekSp@lten#1;vekL@stLine;}
\def\vekL@stLine{vekL@stLine}
\def\xvekSp@lten#1;{\def\temp{#1}%
  \ifx\temp\vekL@stLine
  \else
    \ifnum\@VORNE=1\gdef\@VORNE{0}
    \else\@arraycr\fi%
    #1%
    \expandafter\xvekSp@lten
  \fi}
\makeatother

\newtheorem{thm}{Theorem}[section]
\newtheorem{cor}[thm]{Corollary}
\newtheorem{lem}[thm]{Lemma}

\newtheorem{rem}[thm]{Remark}
\theoremstyle{definition}
\newtheorem{defn}{Definition}[section]

\definecolor{wco}{rgb}{0.5,0.2,0.3}

\numberwithin{equation}{section} \theoremstyle{remark}
\title{{\bf    Long-time behaviour for distribution dependent SDEs with local Lipschitz coefficients }
}
\author{
{\bf     Shan-Shan Hu  }\\
\footnotesize{  Center for Applied Mathematics, Tianjin University, Tianjin 300072, China}\\
\footnotesize{  hushanshan0909@tju.edu.cn}}
\begin{document}
\allowdisplaybreaks

\maketitle
\begin{abstract} 
By using a classical truncated argument and introducing the local Wasserstein distance, the global existence and uniqueness are proved for the distribution dependent SDEs with local Lipschitz coefficients. Due to the measure dependence, the conditions in the sense of pointwise for classical cases can be simplified to the conditions in the sense of integral. On the basis of the well-posedness, we prove the $r$-th moment exponential stability using the measure dependent Lyapunov functions and the existence and uniqueness of invariant probability measure is studied under the integrated strong monotonicity condition. Finally, some examples are given to illustrate the results in this paper.

\end{abstract} \noindent

\section{Introduction}
Distribution dependent SDEs (DDSDEs for short), also called McKean-Vlasov or mean field SDEs, were first studied by Kac \cite{KAC} in the framework of Boltzman equation. Due to the important application in characterizing nonlinear Fokker-Planck equations, DDSDEs have been intensively investigated, see \cite{DV2000,G2002} and references within for Landau type equations. In \cite{WFY2018}, Wang established the strong well-posedness, exponential contraction as well as Harnack inequality. Moreover, for DDSDEs with singular coefficients, gradient estimates and Harnack type inequalities are derived in \cite{HXWFY}. For DDSDEs with H\"{o}lder continuous diffusion coefficient, Bao and Huang discussed the strong well-posedness  and convergence rate of the continuous time Euler-Maruyama scheme corresponding to stochastic interacting particle systems, see \cite{BH}. Furthermore , Song has investigated the gradient estimates and exponential ergodicity for mean-field SDEs with jumps, see \cite{SYM2019} and references therein. Recently, this type SDEs have been applied in \cite{HS,BLPR,RPPWFY} to characterize PDEs involving  Lions derivative ($L$-derivative for short) introduced by P.-L. Lions in his lecture \cite{LNOTES2013} at college de France.

For distribution independent SDEs, long-time behaviour of the solution for SDEs has attracted a great deal of attention, such as the stability and the invariant probability measure. Using measure independent Lyapunov functions to prove moment exponential stability is a classical procedure, which has been studied. In \cite{MBOOK1} and \cite{MBOOK2}, Mao illustrated a number of results on moment exponential stability and almost sure exponential stability of the solution to SDEs, stochastic functional differential equations (SFDEs) and SDEs with markovian switching. 
Moreover, when the probability distribution of the solution converges weakly to some distribution, we call the equation is stable in distribution and the limit distribution is called as a stationary distribution. For highly nonlinear SFDEs where the terms involved the delay components, \cite{IN ME} established the sufficient criteria on the stability in distribution. It is obvious that stability in distribution implies the existence and uniqueness of invariant probability measure. There are also several research papers which were devoted to study of the invariant measure of DDSDEs under a so-called strong monotonicity condition.  

In this paper, we would investigate the long-time behaviour for DDSDEs, including $r$-th moment exponential stability of the solution as well as the existence of the unique invariant probability measure for nonlinear semigroup $P_t^{*}$ associated with DDSDEs. To introduce our results, we first recall the definition of the solution to DDSDEs. Let $\mathscr{P}(\mathbb{R}^d)$ be the space of probability measures on $\mathbb{R}^d$, and $\mathscr{P}_r(\mathbb{R}^d)$ denotes the collection of probability measures with finite $r$-th moments on $\mathbb{R}^d$, here  $r>0$, that is,
$$\mathscr{P}_r(\mathbb{R}^d)=\left\{\mu \in \mathscr{P}(\mathbb{R}^d):\|\mu\|_r^r=\mu(|\cdot|^r):=\int_{\mathbb{R}^d} |x|^r\mu(\text{d}x)<\infty\right\},$$
which is a polish space under the Wasserstein distance
$$\mathbb{W}_r(\mu,\nu):=\inf\limits_{\pi \in \mathscr{C}(\mu,\nu)}\left(\int_{\mathbb{R}^d\times \mathbb{R}^d} |x-y|^r\pi(\text{d}x,\text{d}y)\right)^{\frac{1}{1\vee r}},\;\;\;\mu,\nu\in\mathscr{P}_r,$$
where $\mathscr{C}(\mu,\nu)$ is the set of all couplings  for $\mu$ and $\nu$. In the remainder, we set $r\geq1$. 

Let $W_t$ be an $d$-dimensional Brownian motion defined on a complete probability space with natural filtration $(\Omega,\mathscr{F},\{\mathscr{F}_t\}_{t\geq 0},\mathbb{P})$. For measurable maps
$$b:[0,\infty)\times \mathbb{R}^d \times \mathscr{P}_r(\mathbb{R}^d)\rightarrow \mathbb{R}^d , \;\;\sigma:[0,\infty)\times \mathbb{R}^d \times \mathscr{P}_r(\mathbb{R}^d)\rightarrow \mathbb{R}^d\otimes \mathbb{R}^d ,$$
we consider the following distribution dependent SDEs :
\begin{eqnarray}
\text{d}X_t=b(t,X_t,\mathscr{L}_{X_t})\text{d}t+\sigma(t,X_t,\mathscr{L}_{X_t})\text{d}W_t,\label{1.1}
\end{eqnarray}
where $\mathscr{L}_{X_t}$ denotes the distribution of a random variable $X_t$ on $\mathbb{R}^d$ under $\mathbb{P}$.
\begin{defn}
(1) For any $s\geq 0$, a continuous adapted process $(X_t)_{t\geq s}$ on $\mathbb{R}^d$ is called a (strong) solution of \eqref{1.1} from time $s$, if
$$\int_{s}^{t}\mathbb{E}\left\{|b_r(X_r,\mathscr{L}_{X_r})|+\|\sigma_r(X_r,\mathscr{L}_{X_r})\|_{HS}^2\right\}\text{d}r< \infty$$
and $\mathbb{P}$-a.s.,
$$X_t=X_s+\int_{s}^{t}b_r(X_r,\mathscr{L}_{X_r})dr+\int_{s}^{t}\sigma_r(X_r,\mathscr{L}_{X_r})dW_r,\;\;\;t\geq s.$$
(2) A couple $(\tilde{X}_t,\tilde{W}_t)_{t\geq s}$ is called a weak solution to equation \eqref{1.1} from time $s$, if $\tilde{W}_t$ is the $d$-dimensional Brownian motion with respect to a complete filtration probability space $(\tilde{\Omega},\{\tilde{\mathscr{F}}_t\}_{t\geq 0},\tilde{\mathbb{P}},\tilde{W}_t)$ and $\tilde{X}_t$ solves the DDSDE
$$\text{d}\tilde{X}_t=b(t,\tilde{X}_t,\mathscr{L}_{\tilde{X}_t|\tilde{\mathbb{P}}})\text{d}t+\sigma(t,\tilde{X}_t,\mathscr{L}_{\tilde{X}_t|\tilde{\mathbb{P}}})\text{d}\tilde{W}_t,\;\;\;t\geq s.$$
\end{defn}

The remainder of the paper is organized as follows. In section 2, we establish the well-posedness of DDSDEs under the local Lipschitz condition, which we not only make a truncation on the state variable, but also on the measure. In section 3, measure dependent Lyapunov functions $V$  
are used to investigate the moment exponential stability of the solution. Tools used in proof are the $L$-derivative and It\^{o}'s formula for DDSDEs. In section 4, we provide an integrated strong monotonicity condition to prove the existence and uniqueness of invariant probability measure for $P_t^{*}$; Finally, in section 5, we cite two examples to illustrate the main results.

We will use the following notations frequently:\\
$\diamond$ Denote by  $\partial_t$ the partial differential in time, $\partial_x$ the gradient operator in $x\in \mathbb{R}^d$, and $\partial_{xx}^2$ the Hessian operator in $x\in \mathbb{R}^d$. That is,
$$\partial_x =(\frac{\partial }{\partial x_1},\ldots,\frac{\partial }{\partial x_d}),\partial_{xx}^2=\left(\frac{\partial^2 }{\partial x_i\partial x_j}\right)_{d\times d}.$$\\
$\diamond$ The Hilbert-Schmidt norm of a matrix $A$ is denoted by $\|A\|_{HS}$, which is defined by $\|A\|_{HS}:=\sqrt{\sum\limits_{i,j}a_{i,j}^2}.$\\
\section{Existence and Uniqueness}
As indicated in \cite{WFY2018, HS, HXWFY, BH}, there is a unique strong (weak) solution to \eqref{1.1} under certain assumptions of coefficients  $b$  and  $\sigma$. In this paper, we make the following assumption $\mathbf{(A)}$:

(A1) For any $N\geq 1$, there exists increasing  $C_N\in C([0,\infty);(0,\infty))$ such that
 \begin{align*}
 |b(t,x,\mu)-b(t,y,\nu)|&+\|\sigma(t,x,\mu)-\sigma(t,y,\nu)\|\leq C_N(t)\{|x-y|+\mathbb{W}_r(\mu,\nu)\},
 \end{align*}
$$t\geq0,\;\;|x|\vee|y|\leq N,\;\;\text{supp} \;\mu, \text{supp}\;\nu\subseteq [-N,N]^d.$$

(A2) There exists  increasing  $K\in C([0,\infty);(0,\infty))$  such that
\begin{align*}
 2\langle b(t,x,\mu),x\rangle &+\|\sigma(t,x,\mu)\|^2\leq K(t)\{1+|x|^2+\|\mu\|_r^2\},
 \end{align*}
$$t\geq0,\;\;x,y\in \mathbb{R}^d,\;\;\;\mu\in\mathscr{P}_{r}.$$

(A3) For any bounded sequence $\{\mu^n\}_{n\geq1} \subset \mathscr{P}_r$ with $\mu^n\rightarrow\mu$ weakly as $n\rightarrow\infty$, we have
\begin{eqnarray*}
\lim\limits_{n\rightarrow\infty}\sup\limits_{|x|\leq N}\sup\limits_{t\in[0,T]}\{|b(t,x,\mu^n)-b(t,x,\mu)|+\|\sigma(t,x,\mu^n)-\sigma(t,x,\mu)\|\}=0.
\end{eqnarray*}

(A4) There exist constants $K,\delta>0$, an increasing map $C_N$ such that for any $t\geq0,|x|\vee|y|\leq N,\mu,\nu\in\mathscr{P}_r$,
\begin{align*}
&|b(t,x,\mu)-b(t,y,\nu)|+\|\sigma(t,x,\mu)-\sigma(t,y,\nu)\|\\
&\leq C_N(t)\{|x-y|+\mathbb{W}_{r,N}(\mu,\nu)+Ke^{-\delta C_N}(\mathbb{W}_{r}(\mu,\nu)\wedge1)\},
\end{align*}
 where $\mathbb{W}_{r,N}^r(\mu,\nu):=\inf\limits_{\pi \in \mathscr{C}(\mu,\nu)}\left(\int_{\mathbb{R}^d\times \mathbb{R}^d} { |\phi_N(x)-\phi_N(y)|}^r\pi(\text{d}x,\text{d}y)\right)$ and $\phi_N(x)=\frac{Nx}{N\vee|x|}$.
 \begin{thm}
Assume (A1)-(A4) for some $r\geq1$. Then the DDSDE \eqref{1.1}
has strong existence and uniqueness for every initial distribution in $\mathscr{P}_r$. Moreover, for any $p\geq r$ and  $T>0$, $\mathbb{E}|X_0|^p<\infty$ implies
$$\mathbb{E}\sup\limits_{t\in[0,T]}|X_{t}|^p<\infty.$$
\end{thm}
\begin{proof}
Similar with the classical case, this theorem is proved by a truncation procedure. We construct a series of equations
\begin{eqnarray}
X^{n}(t)=X(0)+\int_{0}^{t}b^n(s,X^{n}_s,\mathscr{L}_{X^{n}_s})\text{d}s+\int_{0}^{t}\sigma^{n}(s,X^{n}_s,\mathscr{L}_{X^{n}_s})\text{d}W_s,\;t\in[0,T].\label{1}
\end{eqnarray}
where $b^{n}(x,\mu)=b(\phi_n(x),\mu\circ \phi^{-1}_n)$,\;\;\;$\sigma^{n}(x,\mu)=\sigma(\phi_n(x),\mu\circ \phi^{-1}_n)$.

As explained in \cite [Proof of Theorem 2.1 (1)] {WFY2018}, we see that for each $n\geq1$, the equation \eqref{1} has a unique solution.
For any $n\geq 1$, setting $\tau^{n}=\inf\{t\geq0:|X^{n}(t)|\geq n\}$ and  applying It\^{o}'s formula to $|X^{n}(t)|^r$ up to time $T\wedge\tau^{n}$, we obtain from (A2) that
$$\mathbb{E}|X_{t\wedge\tau^{n}}^n|^r\leq \mathbb{E}|X_0|^re^{2KrT}=:\delta,\;\;\;n\geq1,t\in [0,T].$$
Consequently, the stopping times
$$\tau^{n}_{N}:=\inf\{t\geq0:|X^{n}(t)|\geq N\},\;\;\;n\geq N\geq 1,$$
satisty
\begin{eqnarray}
\mathbb{P}(\tau^{n}_N<T)\leq \frac{\delta}{N^r}.\label{probability}
\end{eqnarray}
Next, by $(A_1)$  and BDG's inequality, for any $l\geq1$, there exists a constant $C_{n}>0$ such that
\begin{align*}
&\mathbb{E}\sup\limits_{t\in[s,(s+\epsilon)\wedge T]}|X^n(t)-X^n(s)|^{2l}\nonumber\\
&\leq 2^{2l-1}C_n\epsilon^l,\;\;s\in[0,T-\epsilon].
\end{align*}
For any $0<\epsilon<T$, let $k\in\mathbb{N}$ such that $k\epsilon\in[T,T+\epsilon)$. We obtain
\begin{align*}
&\mathbb{E}\sup\limits_{|t-s|\leq \epsilon}|X^n(t)-X^n(s)|^{2l}\nonumber\\
&\leq \sum_{i=1}^{k} \mathbb{E}\sup\limits_{t\in[(i-1)\epsilon,(i\epsilon)\wedge T]}|X^n(t)-X^n((i-1)\epsilon)|^{2l}\nonumber\\
&\leq C_{n,l}(T+\epsilon)\epsilon^{l-1},\;\;\;\;\;\;n\geq 1.
\end{align*}
In particular, taking $l=2$, by Jensen's inequality, we derive
\begin{eqnarray*}
\mathbb{E}\sup\limits_{|t-s|\leq \epsilon}|X^n(t)-X^n(s)|\leq {C_n(T+\epsilon)}^{\frac{1}{4}}\epsilon^{\frac{1}{4}},\;\;\;\;\;\;n\geq 1.
\end{eqnarray*}
We deduce from Arzel\'{a}-Ascoli theorem that $\{\mu^{n}:=\mathscr{L}_{X^{n}}\}_{n\geq1}$ is tight in $\mathscr{P}(C([0,T]))$. By the Prokhorov theorem, there exists a  subsequence $\{n_k\}_{k\geq1}$ and $\mu$ such that $\mu^{n_k}\rightarrow\mu$ weakly in $\mathscr{P}(C[0,T])$ as $k\rightarrow\infty$. Besides, for marginal measures at $t\in[0,T]$, we have $\mu^{n_k}_t\rightarrow\mu_t$  weakly on $\mathbb{R}^{d}$ as $k\rightarrow\infty$.

Let $\tau^{k,l}_N=\tau^{n_k}_N\wedge\tau^{n_l}_N$. From (A1) and (A3), we find a constant $C_N>0$  and a family of constants $\{\epsilon_{k,l}:k,l\geq1\}$ with $\epsilon_{k,l}\rightarrow 0$ as $k,l\rightarrow\infty$ such that
\begin{align*}
&|b^{n_k}(X^{n_k}_{t\wedge\tau^{k,l}_N},\mathscr{L}_{X_t^{n_k}})-b^{n_l}(X^{n_l}_{t\wedge\tau^{k,l}_N},\mathscr{L}_{X_t^{n_l}})|\nonumber\\
\leq& |b(X^{n_k}_{t\wedge\tau^{k,l}_N},\mathscr{L}_{X^{n_k}}\circ\phi_{n_k}^{-1})-b(X^{n_k}_{t\wedge\tau^{k,l}_N},\mu_t)|
+|b(X^{n_k}_{t\wedge\tau^{k,l}_N},\mu_t)-b(X^{n_k}_{t\wedge\tau^{k,l}_N},\mathscr{L}_{X^{n_l}}\circ\phi_{n_l}^{-1})|\nonumber\\
&+|b(X^{n_k}_{t\wedge\tau^{k,l}_N},\mathscr{L}_{X^{n_l}}\circ\phi_{n_l}^{-1})-b(X^{n_l}_{t\wedge\tau^{k,l}_N},\mathscr{L}_{X^{n_l}}\circ\phi_{n_l}^{-1})|\nonumber\\
\leq&C_N|X^{n_k}_{t\wedge\tau^{k,l}_N}-X^{n_l}_{t\wedge\tau^{k,l}_N}|+C_N\epsilon_{k,l}.
\end{align*}
where $\mu_t$ is the limit of weak convergence of the subsequence $\mu^{n_k}$ at time $t$.
Similarly, we also have
\begin{align*}
&\|\sigma^{n_k}(X^{n_k}_{t\wedge\tau^{k,l}_N},\mathscr{L}_{X_t^{n_k}})-\sigma^{n_l}(X^{n_l}_{t\wedge\tau^{k,l}_N},\mathscr{L}_{X_t^{n_l}})\|\nonumber\\
\leq& C_N|X^{n_k}_{t\wedge\tau^{k,l}_N}-X^{n_l}_{t\wedge\tau^{k,l}_N}|+C_N\epsilon_{k,l}.
\end{align*}
By (A1), it follows from BDG's inequality that
\begin{align*}
&\mathbb{E}\sup\limits_{t\in[0,T]}|X^{n_k}_{t\wedge\tau^{k,l}_N}-X^{n_l}_{t\wedge\tau^{k,l}_N}|^2\nonumber\\
&\leq C_N\int_{0}^{T}\mathbb{E}\sup\limits_{0\leq r\leq t}|X^{n_k}_{r\wedge\tau^{k,l}_N}-X^{n_l}_{r\wedge\tau^{k,l}_N}|^2\d t+C_N\epsilon_{k,l}T,\;\;\;k,l\geq N\geq 1,
\end{align*}
and from  Gr\"{o}nwall's inequality that
\begin{eqnarray}
\lim\limits_{k,l\rightarrow\infty}\mathbb{E}\sup\limits_{t\in[0,T]}|X^{n_k}_{t\wedge\tau^{k,l}_N}-X^{n_l}_{t\wedge\tau^{k,l}_N}|^2\leq\lim\limits_{k,l\rightarrow\infty}C_{N}\epsilon_{k,l}Te^{C_N T}=0.\label{expectation}
\end{eqnarray}
Then we infer from \eqref{probability} and \eqref{expectation}  that for any $\epsilon>0$,
\begin{align*}
&\mathbb{P}(\sup\limits_{0\leq t\leq T}|X_t^{n_k}-X_t^{n_l}|\geq\epsilon)\nonumber\\
&\leq \frac{2\delta}{N^r}+\mathbb{P}\left(\sup\limits_{0\leq t\leq T}|X_{t\wedge\tau^{k,l}_N}^{n_k}-X_{t\wedge\tau^{k,l}_N}^{n_l}|\geq\epsilon\right)\\
&\leq \frac{2\delta}{N^r}+\frac{\mathbb{E}\left(\sup\limits_{0\leq t\leq T}|X_{t\wedge\tau^{k,l}_N}^{n_k}-X_{t\wedge\tau^{k,l}_N}^{n_l}|^2\right)}{\epsilon^2}.
\end{align*}

Letting $N\rightarrow\infty$, 
we can find a subsequence $\{\tilde{n}_k\}$ (indexed by $n$ for notational simplicity) and an adapted continuous process $(X_t)_{0\leq t\leq T}$ such that
\begin{eqnarray}
\lim\limits_{n\rightarrow\infty}\sup\limits_{0\leq t\leq T}|X_t^{n}-X_t|=0,\;\;\;\mathbb{P}\text{-a.s.}.\label{0.11}
\end{eqnarray}
Therefore we can pass to the limit in the following equation as $n\rightarrow\infty$,
\begin{eqnarray*}
X^{n}(t)=X(0)+\int_{0}^{t}b^n(X^{n}_s,\mathscr{L}_{X^{n}_s})\text{d}s+\int_{0}^{t}\sigma^{n}(X^{n}_s,\mathscr{L}_{X^{n}_s})\text{d}W_s,\label{approximation}
\end{eqnarray*}
and conclude that $X_t$ solves \eqref{1.1}.

At last, we show that $X_t$ is the unique solution of \eqref{1.1} by condition $(A_4)$. Let $X_t$ and $Y_t$ be two solutions with $X_0=Y_0$. Let
$$\tau_n=\tau_n^{X}\wedge\tau_n^{Y}=\inf\{t\geq0:|X_t|\vee |Y_t|\geq n\},\;\;\;n\geq 1.$$
By It\^{o}'s formula,
\begin{align*}
&\mathbb{E}|X(t\wedge\tau_n)-Y(t\wedge\tau_n)|^2\nonumber\\
\leq& C_n\int_{0}^{t}\mathbb{E}|X(s\wedge\tau_n)-Y(s\wedge\tau_n)|^2\text{d} s+C_n\int_{0}^{t}\mathbb{W}_{r,n}^2(\mathscr{L}_{X_s},\mathscr{L}_{Y_s})\text{d}s+C_n Ke^{-\delta C_n}t\\
\leq &C_n\int_{0}^{t}\mathbb{E}|X(s\wedge\tau_n)-Y(s\wedge\tau_n)|^2\text{d} s +C_n\int_{0}^{t}(\mathbb{E}|X(s\wedge\tau_n)-Y(s\wedge\tau_n)|^r)^{\frac{2}{r}}\text{d} s +C_n Ke^{-\delta C_n}t.
\end{align*}
Applying  Gr\"{o}nwall's inequality and Fatou's lemma, we derive for any $r\in[1,2]$ and $t\in [0,\delta\wedge T]$,
$$\mathbb{E}|X_t-Y_t|^2\leq \liminf\limits_{n\rightarrow\infty}\mathbb{E}|X(t\wedge\tau_n)-Y(t\wedge\tau_n)|^2\leq C_nKte^{-C_n(\delta-t)}=0.\;$$
This implies the pathwise uniqueness up to time $t_0:=\delta\wedge T$. Similarly, we can prove the uniqueness for any $r\in[2,\infty)$ up to time $t_1:=\frac{r\delta}{2}\wedge T$.
\end{proof}
The assumptions (A1)-(A4) can be replaced by the following assumptions (B1)-(B3), and we can make the same well-posedness conclusion.

(B1)For any $N\geq 1$, there exists a constant $C_N>0$ such that
 \begin{align*}
 |b(x,\mu)-b(y,\nu)|&+\|\sigma(x,\mu)-\sigma(y,\nu)\|\leq C_N\{|x-y|+\mathbb{W}_r(\mu,\nu)\},\\
 &|x|\vee|y|\leq N,\;\;\text{supp} \;\mu,\nu\in \mathscr{P}_r.
 \end{align*}

(B2) There exists a constant $K$ such that for any $x\in\mathbb{R}^d,\mu\in\mathscr{P}_r$
\begin{align*}
 &2\langle b(x,\mu),x\rangle\leq K\{1+|x|^2+\|\mu\|_r\},\\
 &\|\sigma(x,\mu)\|^2\leq K\{1+|x|^2+\|\mu\|_r^2\}.
 \end{align*}

(B3) There exist constants $K,\delta>0$, an increasing map $C_\cdot:\mathbb{N}\rightarrow(0,\infty)$ such that for any $x,y\in \mathbb{R}^d$ with $|x|\vee|y|\leq N$, $\mu,\nu\in\mathscr{P}_r$,
\begin{align*}
 &|b(x,\mu)-b(y,\nu)|+\|\sigma(x,\mu)-\sigma(y,\nu)\|\\
 &\leq C_N\{|x-y|+\mathbb{W}_{r,N}(\mu,\nu)+Ke^{-\delta C_N}(\mathbb{W}_{r}(\mu,\nu)\wedge1)\}.
 \end{align*}
\begin{cor}
 Assume (B1)-(B3) holds. Then $\mathscr{L}_{X_0}\in\mathscr{P}_{\theta}(\theta>r)$  implies \eqref{1} has a unique solution $X_t$ with
 $$\mathbb{E}\sup\limits_{t\in[0,T]}|X_{t}|^\theta<\infty,\;\;T>0.$$
\end{cor}
\begin{proof}
We make the same truncation procedure and construct the same approximation equation as \eqref{1}. It suffices to prove the uniform integrability of $\{|X_{\cdot}|^r\}_{n\geq1}$. It is easy to see that if $\sup\limits_{n\geq1}\mathbb{E}[\sup\limits_{0\leq t\leq T}|X^n(t)|^{\theta}]<\infty$, $(\theta>r)$, then $\{|X_{\cdot}|^r\}_{n\geq1}$ is uniform integrabel. It follows from It\^{o}'s formula and (B2) that
\begin{align*}
\text{d} (1+|X^n(t)|^2)^{\frac{\theta}{2}}\leq &C_{k,\theta}(1+|X^n(t)|^2)^{\frac{\theta}{2}-1}[1+|X^n(t)|^2+(1+|X^n(t)|)\|\phi_n(X^n(t))\|_r]\text{d}t\\
&+C_{k,\theta}(1+|X^n(t)|^2)^{\frac{\theta}{2}-1}[1+|X^n(t)|^2+\|\phi_n(X^n(t))\|_r^2]\text{d}t\\
&+C_{\theta}(1+|X^n(t)|^2)^{\frac{\theta}{2}-1}\langle X^n(t),\sigma(\phi_n(X^n(t)),\mathscr{L}_X^n(t)\circ \phi_n^{-1}) \text{d} W_t\rangle,\\
\end{align*}
furthermore, we have
\begin{align*}
\mathbb{E}\sup\limits_{0\leq t\leq T}(1+|X^n(t)|^2)^{\frac{\theta}{2}}\leq &\mathbb{E}(1+|X(0)|^2)^{\frac{\theta}{2}}+C_{k,\theta}\int_{0}^{T}\mathbb{E}(1+|X^n(t)|^2)^{\frac{\theta}{2}}\text{d}t\\
&+C_{k,\theta}\int_{0}^{T}\mathbb{E}(1+|X^n(t)|^2)^{\frac{\theta-1}{2}}(\mathbb{E}(|X^n(t)|^\theta)^{\frac{1}{\theta}}\text{d}t\\
&+C_{k,\theta}\int_{0}^{T}\mathbb{E}(1+|X^n(t)|^2)^{\frac{\theta}{2}-1}(\mathbb{E}(|X^n(t)|^\theta)^{\frac{2}{\theta}}\text{d}t\\
&+C_{\theta}\mathbb{E}\left(\int_{0}^{T}(1+|X^n(t)|^2)^{\theta-1}\|\sigma(\phi_n(X^n(t)),\mathscr{L}_X^n(t)\circ \phi_n^{-1})\|\text{d}t\right)^{\frac{1}{2}}\\
\leq&\mathbb{E}(1+|X(0)|^2)^{\frac{\theta}{2}}+C_{k,\theta}\int_{0}^{T}\mathbb{E}\sup\limits_{0\leq s\leq t}(1+|X^n(s)|^2)^{\frac{\theta}{2}}\text{d}t\\
\leq&\mathbb{E}(1+|X(0)|^2)e^{C_{k,\theta}T}<\infty.
\end{align*}
Thus, $\mu^{n_k}\rightarrow\mu$  in $\mathscr{P}_r(C[0,T])$ as $k\rightarrow\infty$. Besides, for marginal measures at $t\in[0,T]$, we have $\mu^{n_k}_t\rightarrow\mu_t$  weakly in $\mathscr{P}_r(\mathbb{R}^{d})$ as $k\rightarrow\infty$. By (B1) and the similar steps, we derive that $(X_{\cdot}^{n_k})_{n_k\geq1}$ is convergent mutually in probability on $C([0,T])$ and the wellposedness is concluded.
\end{proof}
\section{Moment exponential stability}
Now, we first recall the definition of moment exponential stability, which is standard in literature. Because we will adopt measure dependent Lyapunov functions $V$ and It\^{o}'s formula for DDSDEs, we the definition of Lions derivative for functions on $\mathscr{P}_r(\mathbb{R}^d)$. Here we make use of a straightforward definition introduced in \cite{RPPWFY} to explain the derivative with respect to measure. 
We also introduce briefly original Lions derivative in Appendix. Besides, the following definition of $L$-derivative coincides with the Wasserstein derivative introduced by P.-L. Lions using probability spaces and the lift for functions on $\mathscr{P}_r$.
\begin{defn}
The solution of \eqref{1.1} is said to be $r$-th moment exponentially stable if there is a pair of constants $\lambda$ and $C$ such that
\begin{eqnarray}
\mathbb{E}|X_t|^r\leq C\mathbb{E}|X_s|^re^{-\lambda(t-s)},\;\;\;t\geq s.\label{7.1}
\end{eqnarray}
This means that the $r$-th moment will decrease at most exponentially with exponent $-\lambda$. Recall that $$\limsup\limits_{t\rightarrow\infty}\frac{1}{t}\log(\mathbb{E}|X_t|^r)$$ is the $r$-th moment Lyapunov exponent. Then \eqref{7.1} implies the $r$-th moment Lyapunov exponent should not be greater than $-\lambda$.
\end{defn}

\begin{defn}
(1) A function $V:\mathscr{P}_r(\mathbb{R}^d)\rightarrow \mathbb{R}$ is called $L$-differentiable at $\mu_0\in\mathscr{P}_r(\mathbb{R}^d)$, if the functional
\begin{align*}
L^r(\mathbb{R}^d\rightarrow\mathbb{R}^d,\mu_0)\ni\phi\longmapsto V(\mu_0\circ(\mbox{Id}+\phi)^{-1})
\end{align*}
is Fr\'{e}chet differentiable at $\textbf{0}\in L^r(\mathbb{R}^d\rightarrow\mathbb{R}^d,\mu_0)$, that is, there exists a unique $\xi\in (L^r)^{*}$ such that
\begin{eqnarray}
\lim\limits_{\mu_0(|\phi|^r)\rightarrow 0}\frac{V(\mu_0\circ(\mbox{Id}+\phi)^{-1})-V(\mu_0)-\mu_0(\langle \xi,\phi \rangle)}{\sqrt{\mu_0(|\phi|^r)}}=0.
\end{eqnarray}
In this case, we denote $\partial_\mu V(\mu_0)=\xi$ and call it the $L$-derivative of $V$ at $\mu_0$. If $V$ is $L$-differentiable at all $\mu\in \mathscr{P}_r$, we call it differentiable on $\mathscr{P}_r$.\\
(2) We say that $V\in\mathcal{C}^{(1,1)}(\mathscr{P}_r(\mathbb{R}^d);\mathbb{R})$ if $V$ is $L$-differentiable on $\mathscr{P}_r$ and 

(i) the mapping $\partial_{\mu}V: (\mu,y)\mapsto \partial_{\mu}V(\mu)(y)$ is jointly continuous in $(\mu,y)\in \mathscr{P}_r\times \mathbb{R}^d$;

(ii) for any $\mu \in \mathscr{P}_r(\mathbb{R}^d)$, the mapping $\mathbb{R}^d \ni y\mapsto \partial_{\mu}V(\mu)(y)$ is  continuously differentiable and its derivative $\partial_y\partial_{\mu}V$ is jointly continuous with respect to $(\mu,y)\in \mathscr{P}_r\times \mathbb{R}^d$.\\
(3) We say that $V\in \mathcal{C}^{1,2,(1,1)}([0,\infty) \times \mathbb{R}^d \times \mathscr{P}_r(\mathbb{R}^d);\mathbb{R})$, if $V$ satisfies:

(i) $V(\cdot,\cdot,\mu)\in \mathcal{C}^{1,2}([0,\infty) \times \mathbb{R}^d)$ for each $\mu$;

(ii) $V(t,x,\cdot)$ is in $\mathcal{C}^{(1,1)}(\mathscr{P}_r (\mathbb{R}^d))$ for each $(t,x)$.\\
Moreover all the partial derivatives must be jointly continuous in $(t,x,\mu)$ or $(t,x,\mu,y)$.\\
(4) Finally, the function $V$ is said to be in $\mathscr{C}^{(1,1)}(\mathscr{P}_r(\mathbb{R}^d))$, if \;$V\in \mathcal{C}^{(1,1)}( \mathscr{P}_r(\mathbb{R}^d);\mathbb{R})$ and for any compact set $\mathcal{K}\subset \mathscr{P}_r$,
\begin{eqnarray*}
\sup\limits_{\mu\in \mathcal{K}}\int_{\mathbb{R}^d}\{|\partial_\mu V(\mu)(y)|^r+\|\partial_y\partial_\mu V(\mu)(y)\|_{HS} ^{r}\}\mu(dy)<\infty.
\end{eqnarray*}
The function $V$ is said to be in $\mathscr{C}^{1,2,(1,1)}([0,\infty)\times\mathbb{R}^d\times\mathscr{P}_r)$, if \;$V\in \mathcal{C}^{1,2,(1,1)}([0,\infty)$ $\times \mathbb{R}^d \times \mathscr{P}_r)$ and for any compact set $\mathcal{K}\subset [0,\infty) \times \mathbb{R}^d \times \mathscr{P}_r$,
\begin{eqnarray*}
\sup\limits_{(t,x,\mu)\in \mathcal{K}}\int_{\mathbb{R}^d}\{|\partial_\mu V(t,x,\mu)(y)|^r+\|\partial_y\partial_\mu V(t,x,\mu)(y)\|_{HS} ^{r}\}\mu(dy)<\infty.
\end{eqnarray*}
\end{defn}
\begin{lem}
(It\^{o}'s formula for DDSDEs)
Assume (A1)-(A4), for any $V\in\mathscr{C}^{1,2,(1,1)}([0,\infty)\times\mathbb{R}^d\times\mathscr{P}_r)$, we denote $\mu_t=\mathscr{L}_{X_t}$ and have the following It\^{o}'s formula:
\begin{align*}
&\text{d} V(t,X_t,\mu_t)\\
=&\partial_t V(t,X_t,\mu_t)dt+\langle b, \partial_x V\rangle(t,X_t,\mu_t) \text{d} t+\frac{1}{2}\text{tr}[(\sigma\sigma^{*}\partial_{xx}^2 V)(t,X_t,\mu_t)]\text{d} t\\
&+\displaystyle\int_{\mathbb{R}^d}\langle b(t,y,\mu_t), \partial_\mu V(t,X_t,\mu_t)(y)\rangle \mu_t(\text{d} y)\text{d} t\\ &+\frac{1}{2}\displaystyle\int_{\mathbb{R}^d}\text{tr}[(\sigma\sigma^{*}(t,y,\mu_t)\partial_{y}\partial_{\mu}V(t,X_t,\mu_t)(y)]\mu_t(\text{d} y)\text{d} t\\
&+\langle\partial_x V(t,X_t,\mu_t),\sigma(t,X_t,\mu_t)\text{d}
W_t\rangle.
\end{align*}
\end{lem}
\begin{proof} Let $T>0$ be fixed. According to \cite [Lemma 3.1]{RPPWFY} and \cite [Theorem 3.5]{CDD2014}, it suffices to prove
\begin{eqnarray}
\mathbb{E}\int_{0}^{T}\{|b(t,X_t,\mu_t)|^r+\|\sigma(t,X_t,\mu_t)\|_{HS}^r\}dt<\infty.\label{1.5}
\end{eqnarray}

By (A1) and \eqref{0.11}, it is obvious to see that $X_t$ satisfies \eqref{1.5}.
\end{proof}
We remark that in the case $V\in\mathscr{C}^{(1,1)}(\mathscr{P}_r)$, the above It\^{o} formula reduces to
\begin{align*}
&V(\mu_t)-V(\mu_0)\\
&=\displaystyle\int_{0}^{t}\mathbb{E}\left[\langle{b_s,\partial_\mu V(\mu_s)(X_s)\rangle} +\frac{1}{2}tr\{\sigma_s\sigma_s^*\partial_y\partial_\mu V(\mu_s)(X_s)\}\right]ds\\
&=\displaystyle\int_{0}^{t}\int_{\mathbb{R}^d}\{\sum\limits_{i=1}^{d}b_i(s,y,\mu_s)(\partial_\mu V)_i(\mu_s)(y)+\frac{1}{2}\sum\limits_{i,j,k=1}^{d}(\sigma_{ik}\sigma_{jk})(s,y,\mu_s)\partial_{y_j}(\partial_\mu V)_i(\mu_s)(y)\}\mu_s(dy)ds.
\end{align*}
Based on Lemma 3.1, we introduce a differential operator on $[0,\infty)\times \mathbb{R}^d \times \mathscr{P}_r(\mathbb{R}^d)$ associated with the DDSDE \eqref{1.1}. For any $V\in \mathscr{C}^{1,2,(1,1)}([0,\infty)\times \mathbb{R}^d \times \mathscr{P}_r(\mathbb{R}^d))$, we define $L^{\mu}$ as
\begin{align*}
&(L^{\mu}V)(t,x,\mu)\\
&=\left(\partial_t V+\langle{b,\partial_x V\rangle}(t,x,\mu)+\frac{1}{2}tr\left(\sigma\sigma^{*}(\partial_{xx}^2 V)\right)(t,x,\mu)\right)\\
&\;\;\;\;\;+\int_{\mathbb{R}^d}\left(\langle{b(t,y,\mu),\partial_{\mu}V(t,x,\mu)(y)\rangle}+\frac{1}{2}tr\{(\sigma\sigma^{*})(t,y,\mu)(\partial_y\partial_{\mu}V(t,x,\mu)(y))\}\right)\mu(dy). \end{align*}
\begin{thm}
Assume (A1)-(A4). If there exists a function $V\in \mathscr{C}^{1,2,(1,1)}([0$ $,\infty)\times \mathbb{R}^d \times \mathscr{P}_r)$, and positive constants $C_1,C_2,C_3,\gamma$ such that
\begin{eqnarray}
C_1|x|^r \leq V(t,x,\mu) \leq C_2|x|^r+C_2'\int_{\mathbb{R}^d}|x|^r\mu(\text{d} x),\label{3.2.1}\\
\int_{\mathbb{R}^d}L^{\mu}V(t,x,\mu)\mu (\text{d} x)\leq -\gamma \int_{\mathbb{R}^d}V(t,x,\mu)\mu (\text{d} x),\label{3.1}
\end{eqnarray}and 
\begin{eqnarray}
\sup\limits_{t\geq0}|(L^{\mu} V+\gamma V)(t,x,\mu)|\leq C_3 (|x|^r+\|\mu\|_r^r),\label{3.2}
\end{eqnarray}
then
\begin{eqnarray}
\mathbb{E}|X_t|^r\leq \frac{C_2+C_2'}{C_1}\mathbb{E}|X_0|^r e^{-\gamma t}.\label{2.3}
\end{eqnarray}
\end{thm}
\begin{proof} We may apply  It\^{o}'s formula to $e^{\gamma t}V(t,X_t,\mu_t)$, it follows that
\begin{align*}
&\text{d}\left(e^{\gamma t}V(t,X_t,\mu_t)\right)\\
&=\text{d}\left(e^{\gamma t}V(t,X_t,v)\right)\bigg |_{v=\mu_t}+\text{d}\left(e^{\gamma s}V(s,x,\mu_t)\right)\bigg |_{s=t,x=X_t}\\
&=e^{\gamma t}\{\gamma V(t,X_t,\mu_t)+\partial_tV(t,X_t,\mu_t)+\langle b,\partial_xV\rangle(t,X_t,\mu_t)+\frac{1}{2}tr{(\sigma\sigma^{*}\partial_{xx}^2 V)(t,X_t,\mu_t)}\}\text{d}t\\
&\;\;\;+e^{\gamma t}\int_{\mathbb{R}^d}\left[\langle b(t,y,\mu_t),\partial_{\mu}V(t,X_t,\mu_t)(y)\rangle+\frac{1}{2}tr\{\sigma\sigma^{*}(t,y,\mu_t)\partial_y\partial_\mu V(t,X_t,\mu_t)(y)\}\mu_t(dy)\right]\text{d}t\\
&\;\;\;+e^{\gamma t}\langle \partial_x V(t,X_t,\mu_t),\sigma(t,X_t,\mu_t)\text{d}W_t\rangle\\
&=e^{\gamma t}\left(L^{\mu_t}V+\gamma V\right)(t,X_t,\mu_t)\text{d}t+e^{\gamma t}\langle \partial_x V(t,X_t,\mu_t),\sigma(t,X_t,\mu_t)\text{d}W_t\rangle.
\end{align*}
For each positive integer $n$, $\tau_n$:= $ \inf\{t\geq 0:|X_t|\geq n\}$, obviously, $\tau_n\rightarrow \infty$ as $n\rightarrow \infty$ almost surely, then we can derive that
\begin{align*}
e^{\gamma(t\wedge\tau_n )}V(t\wedge\tau_n,X(t\wedge\tau_n),\mu_s)|{_{s=t\wedge\tau_n}}&=V(0,X_0,\mu_0)\\
&\;\;\;\;+\int_{0}^{t\wedge\tau_n }[e^{\gamma s}(L^{\mu_s}V+\gamma V)(s,X_s,\mu_s)]\text{d}s\\
&\;\;\;\;+\int_{0}^{t\wedge\tau_n }e^{\gamma s}\langle \partial_x V(s,X_s,\mu_s),\sigma(s,X_s,\mu_s)\text{d}W_s\rangle.
\end{align*}
Taking the expectation on both sides of the above equality and by \eqref{3.2.1}, \eqref{3.1}, \eqref{3.2}, we deduce that
\begin{align}
&\mathbb{E}\{e^{\gamma(t\wedge\tau_n)}V(t\wedge\tau_n,X(t\wedge\tau_n),\mu_s)|_{s=t\wedge\tau_n}\}\nonumber\\
&=\mathbb{E}V(0,X_0,\mu_0)+\int_{0}^{t}\mathbb{E}e^{\gamma s}(L^{\mu_s}V+\gamma V)(s,X_s,\mu_s)\text{d}s-\mathbb{E}\int_{t\wedge\tau_n}^{t}e^{\gamma s}(L^{\mu_s}V+\gamma V)(s,X_s,\mu_s)\text{d}s\nonumber\\
&\leq \mathbb{E}V(0,X_0,\mu_0)+\mathbb{E}\left[\mathbb{I}_{\{t>\tau_n\}}\int_{\tau_n}^{t}e^{\gamma s}(L^{\mu_s}V+\gamma V)(s,X_s,\mu_s)\text{d}s\right]\nonumber\\
&\leq \mathbb{E}V(0,X_0,\mu_0)+C_3\mathbb{E}\left[\mathbb{I}_{\{t>\tau_n\}}\int_{0}^{t}e^{\gamma s} (|X_s|^r+\mu_s(|\cdot|^r))\text{d}s\right],\label{3.3}
\end{align}
where $\mathbb{I}_{\{t>\tau_n\}}$ is the indicator function.
On the other hand, due to $e^{\gamma(t\wedge\tau_n)}>0$, it follows that
\begin{align}
\mathbb{E}\{e^{\gamma(t\wedge\tau_n)}V(t\wedge\tau_n,X(t\wedge\tau_n),\mu_s)|_{s=t\wedge\tau_n}\}&\geq \mathbb{E}\{e^{\gamma(t\wedge\tau_n)}(C_1|X(t\wedge\tau_n )|^r\}\label{2.5},
\end{align}
Which, together with \eqref{3.3}, implies the desired inequality \eqref{2.3}  by letting $n\rightarrow \infty$.
\end{proof}
\begin{rem}
If the Lyapunov function only dependent on time $t$ and measure $\mu$, then the condition \eqref{3.2.1} can be reduced by
\begin{eqnarray}
C_1\int_{\mathbb{R}^d}|x|^r\mu(\text{d}x) \leq \int_{\mathbb{R}^d}V(t,\mu)\mu(\text{d}x) \leq C_2\int_{\mathbb{R}^d}|x|^r\mu(\text{d} x).
\end{eqnarray}
\end{rem}
\begin{rem}
If the condition \eqref{3.2} is not satisfied, then we may derive the same estimates as before by lifting the integrality of $X_0$.
\end{rem}
\section{Invariant probability measures}
In the following part, we investigate a type of long-time behaviour for the solution of DDSDEs, that is, the existence and uniqueness of invariant probability measure $\mu\in \mathscr{P}_{r}$. For any $\mu\in \mathscr{P}_r$ and $s\geq 0$, let $(X_{s,t}^{\mu})$ solve \eqref{1.1} from $s$ with $\mathscr{L}_{X_{s,s}}=\mu$. When the DDSDE \eqref{1.1} has strong uniqueness (weak uniqueness is also workable), define
\begin{eqnarray*}
P_{s,t}^{*}\mu=\mathscr{L}_{X_{s,t}},\;\;s\leq t,\;\;\mu\in\mathscr{P}_r(\mathbb{R}^d).
\end{eqnarray*}

As shown in \cite{WFY2018} that $P_{s,t}^{*}$ is a nonlinear semigroup satisfying
\begin{eqnarray*}
P_{s,t}^{*}=P_{r,t}^{*}P_{s,r}^{*},\;\;0\leq s\leq r\leq t.
\end{eqnarray*}
When $\sigma_t$ and $b_t$ do not depend on time $t$, we have $P_{s,t}^{*}=P_{t-s}^{*}$ for $0\leq s\leq t$. We call $\mu \in \mathscr{P}_r$ an invariant probability measure of $P_t^{*}$ if $P_t^{*}\mu=\mu$ for all $t\geq0$. 
It is important to point out that the standard Krylov-Bogoliubov procedure on invariant measure cannot be applied to DDSDEs. It is because of the nonlinearity of semigroup $P_{s,t}^{*}$ that the classical tightness argument does not work.

In \cite{WFY2018}, sufficient conditions are given for $\mathbb{W}_2$-exponential contraction of $P_t^{*}$, the existence and uniqueness of invariant measure $\mu\in\mathscr{P}_2$. Now, we intend to weaken the condition \textbf{(H2$'$)} in \cite{WFY2018} into the following integrable strong monotonicity \textbf{(H)}:\\
\textbf{(H)} There exists a positive constant $C_1$ such that for any $\pi\in \mathscr{C}(\mu,\nu)$,
\begin{equation}
\begin{split}
&\int_{\mathbb{R}^d\times\mathbb{R}^d}\{\langle b(x,\mu)-b(y,\nu),x-y\rangle+\|\sigma(x,\mu)-\sigma(y,\nu)\|_{HS}^2\}\pi(\text{d}x,\text{d}y)\\
&\leq\;-\frac{C_1}{2}\int_{\mathbb{R}^d\times\mathbb{R}^d}|x-y|^2\pi(\text{d}x,\text{d}y),\;\;x,y\in \mathbb{R}^d;\mu,\nu \in \mathscr{P}_2.\label{6.3}
\end{split}
\end{equation}
\begin{thm}
Let $b_t=b$, $\sigma_t=\sigma$ do not depend on time $t$. Assume (A1)-(A4)  for $\theta=2$ and \textbf{(H)} hold, then $P_t^{*}$ has a unique invariant probability measure $\mu\in\mathscr{P}_2$.
\end{thm}
\begin{proof} Let $\delta_0$ be the Dirac measure at point $0\in\mathbb{R}^d$, then $P_t^{*}\delta_0=\mathscr{L}_{X_t(0)}$. 
As explained in \cite [Proof of Theorem 3.1 (2)] {WFY2018}, it suffices to prove
${P_t^{*}}$ is a $\mathbb{W}_2$-Cauchy family. To this end, we estimate $\mathbb{W}_2(P_t^{*}\delta_0,P_{t+s}^{*}\delta_0)$ and $\mathbb{E}|X_s(0)|^2$ for the purpose of proving $\mathbb{W}_2$-exponential contraction and the bounded property of $\mathbb{E}|X_s(0)|^2$ for any $s\geq0$.

(1) Let $X_t$ and $Y_t$ be two solutions of \eqref{1.1} such that $\mathscr{L}_{X_0}=\delta_0,\mathscr{L}_{Y_0}=P_s^{*}\delta_0 (s\geq0)$ and
$\mathbb{E}|X_0-Y_0|^2=\mathbb{W}_2(P_s^{*}\delta_0,\delta_0)^2$, these means $\mathbb{E}\sup\limits_{0\leq t\leq T}|X_t|^4<\infty$ and $\mathbb{E}\sup\limits_{0\leq t\leq T}|Y_t|^4<\infty$ for any $T\geq t\geq 0$.

By It\^{o}'s formula and taking expectation, we have
\begin{align*}
\mathbb{E}|X_t-Y_t|^2&=\mathbb{E}|X_0-Y_0|^2+2\int_{0}^{t}\mathbb{E}\langle X_s-Y_s,\{b(X_s,\mu_s)-b(Y_s,\nu_s)\}\text{d}s\rangle\nonumber\\
&\;\;\;\;+\int_{0}^{t}\mathbb{E}\|\sigma(X_s,\mu_s)-\sigma(Y_s,\nu_s)\|_{HS}^2\text{d}s,
\end{align*}
Taking \eqref{6.3} into account,
$$\int_{\mathbb{R}^d\times\mathbb{R}^d}|X_t(\omega_1)-Y_t(\omega_2)|^2\mathbb{P}_{X_t,Y_t}(\text{d}\omega_1,\text{d}\omega_2)\leq \mathbb{E}|X_0-Y_0|^2-C_1\int_{0}^{t}\mathbb{W}_2(\mu_s,\nu_s)^2\text{d}s.$$
where $\mathbb{P}_{X_t,Y_t}$ is the joint probability measure of $X_t$ and $Y_t$. Noting that $\mathbb{W}_2(\mu_s,\nu_s)^2\leq\mathbb{E}|X_s-Y_s|^2$, by Gronwall's lemma, this implies
\begin{eqnarray*}
\mathbb{E}|X_t-Y_t|^2\leq\mathbb{E}|X_0-Y_0|^2e^{-C_1t},\;\;\;\;t\geq 0,
\end{eqnarray*}
that is to say,
\begin{eqnarray}
\mathbb{W}_2(P_t^{*}\delta_0,P_{t+s}^{*}\delta_0)\leq \mathbb{W}_2(P_s^{*}\delta_0,\delta_0)e^{-C_1t},\;\;\;\;t\geq0.\label{3.6}
\end{eqnarray}

By \textbf{(H)} with $y=0, \nu=\delta_0$, taking $\pi=\mu\times\delta_0$, which means $\pi(\text{d}x,\text{d}y)=\mu(\text{d}x)\times\delta_0$. It is easy to see that
\begin{align}
&\int_{\mathbb{R}^d}\{2\langle b(x,\mu),x\rangle+\|\sigma(x,\mu)\|_{HS}^2\}\mu(\text{d}x)\nonumber\\
&\leq \int_{\mathbb{R}^d\times\mathbb{R}^d}\{2\langle b(x,\mu)-b(y,\delta_0),x\rangle+2\|\sigma(x,\mu)\|_{HS}^2\}\pi(\text{d}x,\text{d}y)\nonumber\\
&\;\;\;\;+\int_{\mathbb{R}^d}\{2\langle b(0,\delta_0),x\rangle+2\|\sigma(0,\delta_0)\|_{HS}^2\}\mu(\text{d}x)\nonumber\\
&\leq -(C_1-\epsilon)\mu(|\cdot|^2)+C_2\label{5.5}
\end{align}
holds for some constants $0<\epsilon<C_1$ and $C_2>0$. By It\^{o}'s formula, combing with \eqref{5.5}
, we derive
$$\mathbb{E}|X_t(0)|^2<\frac{C_2-C_2e^{-(C_1-\epsilon)t}}{C_1-\epsilon},\;\;t\geq0$$
and
\begin{eqnarray}
\sup\limits_{t\geq0}\mathbb{E}|X_t(0)|^2<\infty.\label{3.7}
\end{eqnarray}

Hence, \label{3.6} and \label{3.7} and semigroup property guarantee that
\begin{equation}
\begin{split}
\lim\limits_{t\rightarrow\infty}\sup\limits_{s\geq0}\mathbb{W}_2(P_t^{*}\delta_0,P_{t+s}^{*}\delta_0)^2&\leq \lim\limits_{t\rightarrow\infty}\sup\limits_{s\geq0} \mathbb{W}_2(P_s^{*}\delta_0,\delta_0)e^{-C_1t}\\
&=\lim\limits_{t\rightarrow\infty}\sup\limits_{s\geq0}\mathbb{E}|X_s(0)|^2e^{-C_1t}\\
&=0,\label{3.10}
\end{split}
\end{equation}
which implies that ${P_t^{*}}$ is a $\mathbb{W}_2$-Cauchy family and
\begin{eqnarray}
\lim\limits_{t\rightarrow\infty}\mathbb{W}_2(P_t^{*},\mu)=0\label{3.8}
\end{eqnarray}
holds for some $\mu\in\mathscr{P}_2$. We claim that $\mu$ is an invariant probability measure for ${P_t^{*}}$.

(2) Now, we intend to prove $\mathbb{W}_1$-exponential contraction of $P_t^{*}\mu$ and $P_t^{*}P_s^{*}\delta_0$. Let $X_t$ and $Y_t$ be two solutions of \eqref{1.1} such that $\mathscr{L}_{X_0}=\mu,\mathscr{L}_{Y_0}=P_s^{*}\delta_0 (s\geq0)$ and
$$\mathbb{E}|X_0-Y_0|=\mathbb{W}_1(P_s^{*}\delta_0,\mu).$$
Similarly, by It\^{o}'s formula,
\begin{align*}
\mathbb{E}(\epsilon+|X_t-Y_t|^2)^{\frac{1}{2}}&\leq e^{-C_1t}\mathbb{E}(\epsilon+|X_0-Y_0|^2)^{\frac{1}{2}}\nonumber\\
&\;\;\;+e^{-C_1t}\int_{0}^{t}e^{C_1s}\{C_1\mathbb{E}(\epsilon+|X_s-Y_s|^2)^{\frac{1}{2}}-{\frac{C_1}{2\sqrt{\epsilon}}}\mathbb{E}(\epsilon+|X_s-Y_s|^2)\}\text{d}s\nonumber\\
&\leq e^{-C_1t}\mathbb{E}(\epsilon+|X_0-Y_0|^2)^{\frac{1}{2}}+\frac{3C_1\sqrt{\epsilon}}{2}e^{-C_1t}\int_{0}^{t}e^{C_1s}\text{d}s\nonumber\\
&\leq e^{-C_1t}\mathbb{E}(\epsilon+|X_0-Y_0|^2)^{\frac{1}{2}}+\frac{3\sqrt{\epsilon}}{2}.
\end{align*}
Let $\epsilon\rightarrow0$, we derive
\begin{eqnarray}
\mathbb{W}_1(P_t^{*}\mu,P_t^{*}P_s^{*}\delta_0)\leq e^{-C_1t}\mathbb{W}_1(P_s^{*}\delta_0,\mu).\label{3.9}
\end{eqnarray}
Combing this with \eqref{3.10} and  \eqref{3.8}, we obtain
\begin{align*}
\mathbb{W}_1(P_s^{*}\mu,\mu)&\leq\mathbb{W}_1(P_s^{*}\mu,P_s^{*}P_t^{*}\delta_0)+\mathbb{W}_1(P_s^{*}P_t^{*}\delta_0,P_t^{*}\delta_0)+\mathbb{W}_1(P_t^{*}\delta_0,\mu)\\ &\leq\mathbb{W}_1(\mu,P_t^{*}\delta_0)e^{-C_1s}+\mathbb{W}_2(P_s^{*}P_t^{*}\delta_0,P_t^{*}\delta_0)+\mathbb{W}_2(P_t^{*}\delta_0,\mu),
\end{align*}
then $\mu$ is a invariant probability measure by letting $t\rightarrow\infty$.
\end{proof}
\begin{rem}
This theorem is a type of generalization of the main results in \cite{AX} and the theorem 3.1 in \cite{WFY2018}.
\end{rem}
We'd like to illustrate  by taking the following example
\begin{eqnarray}
\text{d}X_t=[-\alpha X_t-\mathbb{E}X_t]\text{d}t+\text{d}W_t,\;\;\;\;t\geq0,\label{4.5}
\end{eqnarray}
where $W$ is a standard one-dimensional Browian motion; $\alpha$ is a positive constant; $\mathbb{E}(X_t)$ is the mean of $X_t$. Since \eqref{4.5} is a gradient system, we may obtain invariant measures have explicit expression, i.e.
\begin{eqnarray}
\mathbb{P}_m(\text{d}x)=\frac{1}{Z}\exp\{-(\alpha x^2+2mx)\}\text{d}x,\label{4.6}
\end{eqnarray}
where $Z$ is the normalizing constant, and the constant must satisfy the self-consistence equation $m=\frac{m}{\alpha}$. Combing this with \eqref{4.6}, system \eqref{4.5} will have a unique invariant measure if $\alpha>0$.

However, according to in \cite[theorem 3] {AX}, we only can make the conclusion that system \eqref{4.5} will have a unique invariant measure if $\alpha$ satisfies $\alpha^2>6$. This indicates that the condition in theorem 3 is not a necessary condition for the existence of a invariant measure. Besides, by checking the condition \textbf{(H2$'$)} in \cite{WFY2018}, we can only get system \eqref{4.5} will have a unique invariant measure if $\alpha\geq1$ and the method is unworkable if $0<\alpha\leq1$. Fortunately, it is obvious that system \eqref{4.5} satisfy the condition \textbf{(H)} if $\alpha>0$, which implies \eqref{4.5} have a unique invariant measure $\mu\in\mathscr{P}_2$.

\section{Examples}
\subsection{DDSDE with local Lipschitsz condition}
We consider the following 1-dimension equations
\begin{eqnarray}
\text{d}X_t=[-(X_t)^3-X_t\int_{\mathbb{R}^d}((L\vee|x|)\wedge M)\mu_t(\text{d}x)]\text{d}t+\frac{1}{2}X_t \text{d}W_t,\label{5.1}
\end{eqnarray}
where $L$ and $M$ are positive constants.\\
(1) Since $b(x,\mu)=-x^3-x\int_{\mathbb{R}^d}((L\vee|x|)\wedge M)\mu(\text{d}x)$, $\sigma(x)=\frac{1}{2}x$ grow in polynomial with respect to $x$ and $b$ is linear with respect to measure, then the (A1)-(A4) hold true and the well-posedness follows from Theorem 2.1.\\
(2) To show the moment exponential stability, we may choose the Lyapunov function as
$$V(\mu)=\int_{\mathbb{R}}|x|^4\mu(\text{d}x),$$
by simple calculation, we have
\begin{align*}
L^{\mu}V(\mu)&=-4\int_{\mathbb{R}}y^6\mu(\text{d}y)-4\int_{\mathbb{R}}y^4\mu(\text{d}y)\int_{\mathbb{R}}(L\vee|x|)\wedge M)\mu(\text{d}x)+\frac{3}{2}\int_{\mathbb{R}}y^4\mu(\text{d}y)\\
&\leq -4(M-\frac{3}{8})\int_{\mathbb{R}}y^4\mu(\text{d}y)\\
&=-4(M-\frac{3}{8})V(\mu).
\end{align*}
Consequently, the solution of the equation \eqref{5.1} is 4-th moment exponentially stable for $M>\frac{3}{8}$ and the Lyapunov exponent is not greater than $-4(M-\frac{3}{8})$.\\
(3) As for the invariant probability measure, we may prove the existence and uniqueness. Indeed, we observe that for any $\pi\in \mathscr{C}(\mu,\nu)$,
\begin{align*}
&\int_{\mathbb{R}\times \mathbb{R}}\langle b(x,\mu)-b(y,\nu), x-y\rangle+\|\sigma(x)\|^2\pi(\text{d}x,\text{d}y)\\
&\leq -\left(\frac{3N-M}{2}-\frac{1}{4}\right)\int_{\mathbb{R}\times \mathbb{R}}|x-y|^2\pi(\text{d}x,\text{d}y),
\end{align*}
which implies the assertion by Theorem 4.1.

\subsection{Example due to Landau type equations}
Below, we consider the case with Maxwell molecules $\gamma=0$ for Landau type DDSDEs.\\
For two Lipschitz continuous maps:\\
$$b_0:\mathbb{R}^d\rightarrow \mathbb{R}^d,\sigma_0:\mathbb{R}^d \rightarrow \mathbb{R}^d\otimes\mathbb{R}^d.$$
Let\\
$$b^{\alpha}(x,\mu):=\int_{\mathbb{R}^d}b_0(x-\alpha z)\mu(\text{d}z),\sigma^{\alpha}(x,\mu):=\int_{\mathbb{R}^d}\sigma_0(x-\alpha z)\mu(\text{d}z)$$
\;\;\;\;\;\;\;\;\;\;\;\;\;$$\alpha \in \mathbb{R},x\in \mathbb{R}^d,\mu \in \mathscr{P}_{r}.$$\\
For fixed $\alpha \in \mathbb{R}$, consider the SDE\\
$$dX_t=b^{\alpha}(X_t,\mathscr{L}_{X_t})\text{d}t+{\sigma}^{\alpha}(X_t,\mathscr{L}_{X_t})\text{d}W_t.$$
In \cite{WFY2018}, Wang discussed the strong well-posedness  for this type of equation. Besides, gradient estimates and exponential ergodicity for this type of SDEs with jumps have been investigated in \cite{SYM2019}.
\subsubsection*{Example 2:}
\begin{eqnarray}
\text{d}X_t=-2\left(\int_{\mathbb{R}}(X_t+\alpha y)\mathscr{L}(X_t)(\text{d}y)\right)\text{d}t+\left(\int_{\mathbb{R}}(X_t+\alpha y)\mathscr{L}(X_t)(\text{d}y)\right)\text{d}W_t.\label{3.5}
\end{eqnarray}
First of all, we point that  the equation \eqref{3.5} has a unique solution $(X_t)_{t\geq 0}$ for any $\mathbb{E}|X_0|^4<\infty$, and $\mathbb{E}\sup\limits_{t\in [0,T]}|X_t|^{4}< \infty$ for all $T>0$. \\
(1) Choosing $V(x,\mu)=\left(\displaystyle\int_{\mathbb{R}}(x^2+\alpha y^2)\mu(\text{d}y)\right)^2$ for Lyapunov function. By the definition of $L$-derivative, we have
$$\partial_{x}V(x,\mu)=4x\left(\int_{\mathbb{R}}(x^2+\alpha y^2)\mu(\text{d}y)\right),\;\;\;\;\;\;\partial_{xx}^2V(x,\mu)=4 \left(\int_{\mathbb{R}}(x^2+\alpha y^2)\mu(\text{d}y)\right)+8x^2,$$
$$\partial_{\mu}V(x,\mu)(z)=4\alpha z\left(\int_{\mathbb{R}}(x^2+\alpha y^2)\mu(\text{d}y)\right),\;\;\;\partial_z \partial_{\mu}V(x,\mu)(z)=4\alpha \left(\int_{\mathbb{R}}(x^2+\alpha y^2)\mu(\text{d}y)\right).$$
By simple calculation, we deduce that
\begin{align*}
(L^{\mu}V)(x,\mu) &=-8x\left(\int_{\mathbb{R}}(x+\alpha y)\mu(\text{d}y)\right)\cdot \left(\int_{\mathbb{R}}(x^2+\alpha y^2)\mu(\text{d}y)\right)\\
&\;\;\;\;+\left(\int_{\mathbb{R}}(x+\alpha y)\mu(\text{d}y)\right)^2\left[2\left(\int_{\mathbb{R}}(x^2+\alpha y^2)\mu(\text{d}y)\right)+4x^2\right]\\
&\;\;\;\;-8\alpha\left(\int_{\mathbb{R}}(x^2+\alpha y^2)\mu(\text{d}y)\right)\cdot \left(\int_{\mathbb{R}}\left(z\int_{\mathbb{R}}(z+\alpha y)\mu(\text{d}y)\right)\mu(\text{d}z)\right)\\
&\;\;\;\;+2\alpha\left(\int_{\mathbb{R}}(x^2+\alpha y^2)\mu(\text{d}y)\right)\cdot \left(\int_{\mathbb{R}}\left(\int_{\mathbb{R}}(z+\alpha y)\mu(\text{d}y)\right)^2\mu(\text{d}z)\right)\\
&=:-8i_1+2i_2+4i_3-8\alpha i_4+2\alpha i_5.
\end{align*}
Let
$$a_1=\int_{\mathbb{R}}x^1\mu(\text{d}x), \ a_2=\int_{\mathbb{R}}x^2\mu(\text{d}x), \ a_3=\int_{\mathbb{R}}x^3\mu(\text{d}x), \ a_4=\int_{\mathbb{R}}x^4\mu(\text{d}x).$$
$$A_1=\left(\int_{\mathbb{R}}x\mu(\text{d}x)\right)^1, A_2=\left(\int_{\mathbb{R}}x\mu(\text{d}x)\right)^2, A_3=\left(\int_{\mathbb{R}}x\mu(\text{d}x)\right)^3, A_4=\left(\int_{\mathbb{R}}x\mu(\text{d}x)\right)^4.$$
Furthermore,\\
\begin{align*}
&\int_{\mathbb{R}}(L^{\mu}V)(x,\mu)\mu(\text{d}x)\\
&=-8\int_{\mathbb{R}}i_1\mu(\text{d}x)+2\int_{\mathbb{R}}i_2\mu(\text{d}x)+4\int_{\mathbb{R}}i_3\mu(dx)-8\alpha\int_{\mathbb{R}}i_4\mu(\text{d}x)+2\alpha\int_{\mathbb{R}}i_5\mu(\text{d}x)\\
&=:-8I_1+2I_2+4I_3-8\alpha I_4+2\alpha I_5\\
&=-2a_4-(8\alpha ^2+6\alpha)a_2{^2}+4\alpha a_1a_3+(2\alpha^3-2\alpha^2)a_1{^2}a_2.
\end{align*}
Here,
$$I_1=-8\left(a_4+\alpha a_2{^2}+\alpha a_1a_3+\alpha ^2 a_1{^2}a_2\right),
I_2=2\left(a_4+\alpha a_2+(\alpha^3+3\alpha^2)a_1{^2}a_2+2\alpha a_1a_3\right),$$
$$I_3=4\left(a_4+\alpha ^2 a_1{^2}a_2+2\alpha a_1a_3\right),
I_4=-8\alpha\left((1+\alpha) a_2{^2}+\alpha a_1{^2}a_2\right),
I_5=2\alpha\left(4 a_2{^2}+\alpha a_1{^2}a_2\right).$$
Above all, for any $\alpha \in (0,\frac{1}{2}),\exists \gamma >0$, such that $ 0< \gamma \leq (2-4\alpha)$,
$$\int_{\mathbb{R}}(L^{\mu}V)(x,\mu)\mu(\text{d}x) \leq -\gamma\int_{\mathbb{R}}V(x,\mu)\mu(\text{d}x). $$
Therefore, we obtain the assertion that the solution of equation is $4$-th moment exponentially stable and the $4$-th moment Lyapunov exponent should not be greater than $-2+4 \alpha$. \\
(2) If $\alpha\in(0,\frac{1}{2})$, for any $\pi\in \mathscr{C}(\mu,\nu)$, we derive
\begin{align*}
&\int_{\mathbb{R}^2}\{2\langle b(x,\mu)-b(y,\nu),x-y\rangle+\|\sigma(x,\mu)-\sigma(y,\nu)\|_{HS}^2\}\pi(\text{d}x,\text{d}y)\\
&\leq-3\int_{\mathbb{R}^2}|x-y|^2\pi(\text{d}x,\text{d}y)-(2\alpha-1)\left(\int_{\mathbb{R}^2}(x-y)\pi(\text{d}x,\text{d}y)\right)^2\\
&\leq-(2-2\alpha)\int_{\mathbb{R}^2}|x-y|^2\pi(\text{d}x,\text{d}y).
\end{align*}
By Theorem 4.1, it is obvious that DDSDE \eqref{3.5} satisfy \textbf{(H2$'$)} and constants $C_2>C_1\geq0$, then the semigroup $P_t^{*}$ associated with DDSDE \eqref{3.5} has a unique invariant probability measure.
\section*{Acknowledgements}
The author would like to thank my supervisor P.f. Feng-Yu Wang and P.f. Jianhai Bao for corrections and helpful comments.

\end{document}